\documentclass[a4paper,11pt]{amsart}
\usepackage{amsmath}
\usepackage{multirow}
\usepackage{amssymb}
\usepackage{amsaddr}
\usepackage{hyperref}
\usepackage{cite}
\usepackage{graphicx}
\usepackage{subfig}
\usepackage{xspace}
\usepackage{todonotes}
\usepackage{mathtools}
\usepackage{xr}
\usetikzlibrary{calc}
\tikzset{nodeStyle/.style = {circle,draw,minimum size=30pt}}
\tikzset{arrowStyle/.style = {-latex}}

\newcommand{\Tr}{\mathrm{Tr}}
\newcommand{\tr}{\Tr}

\newcommand{\1}{{\rm 1\hspace{-0.9mm}l}}
\newcommand{\Id}{\1}

\newcommand{\dd}{\mathrm{d}}
\newcommand{\ee}{\mathrm{e}}

\newcommand{\Z}{\mathbb{Z}}

\newcommand{\R}{\mathbb R}

\newcommand{\ie}{\emph{i.e.\/}}

\newtheorem{theorem}{Theorem}

\newtheorem{lemma}[theorem]{Lemma}

\newtheorem{proposition}[theorem]{Proposition}
\newtheorem{remark}[theorem]{Remark}

\newcommand{\ER}{Erd{\H o}s-R\'enyi\xspace}
\newcommand{\WS}{Watts-Strogatz\xspace}
\newcommand{\BA}{Barab{\'a}si-Albert\xspace}
\newcommand{\CL}{Chung-Lu\xspace}

\title{Asymptotic entropy of the Gibbs state of complex networks}
\author{Adam Glos$^{1,2,*}$\and 
Aleksandra Krawiec$^1$ \and {\L}ukasz Pawela$^{1}$}

\address{$^1$Institute of Theoretical and Applied Informatics, 
	Polish Academy of Sciences, ul. Ba{\l}tycka 5, 44-100 Gliwice, Poland}
\address{$^2$Institute of Informatics, Silesian University of Technology, \\
	ul. Akademicka 16, 44-100 Gliwice, Poland}
\begin{document}
\maketitle
$^{*}$ \normalsize\url{aglos@iitis.pl} 

\begin{abstract}
In this work we study the entropy of the Gibbs state corresponding to a graph.
The Gibbs state is obtained from the Laplacian, normalized Laplacian or
adjacency matrices associated with a graph. We calculated the entropy of the
Gibbs state for a few classes of graphs and studied their behavior with changing
graph order and temperature. We illustrate our analytical results with numerical
simulations for \ER, \WS, \BA and \CL graph models and a few real-world graphs.
Our results show that the behavior of Gibbs entropy as a function of the
temperature differs for a choice of real networks when compared to the random
\ER graphs.
\end{abstract}

\section{Introduction} A network represents a relationship among units of a
complex system. The relations are encoded by edges while units are associated
with nodes. Typical random graph models such as \ER{} graphs
\cite{erdHos1960evolution} are usually not suitable for modeling real-world
networks like the Internet \cite{albert2002statistical}. Here complex network
theory comes as a possible remedy. The boundary between a graph and a network is
rather blurred, nevertheless a typical network is scale-free, small-world and
has social structures. Typical examples of complex networks are Watts-Strogatz
\cite{watts1998collective} and Barab\'asi-Albert networks
\cite{albert2002statistical}.

Graph entropy describes the graph in the context of evolution on it
\cite{de2016spectral}. In classical walks one typically considers the von
Neumann entropy calculated for the Laplacian, as Laplacian defines valid
continuous-time stochastic evolution
\cite{kenkre1973generalized,childs2004spatial}. Studies on various types of
graph entropy can be found in the literature~\cite{braunstein2006laplacian}. The
von Neumann entropy for complex networks was analyzed in
\cite{anand2011shannon,nicolini2018thermodynamics}. Thermal state entanglement
entropy on quantum graphs was studied in \cite{verga2019thermal}. Entropy
measure for complex networks using its Gibbs state was defined in
\cite{de2016spectral}.

In contrary to stochastic evolution, continuous-time quantum walks accept
arbitrary symmetric graph matrix which for undirected graphs includes adjacency
matrix and normalized Laplacian
\cite{childs2004spatial,wong2016laplacian,glos2018vertices}. Since it is known
that the choice of a graph matrix does affect the evolution of quantum walk
\cite{wong2016laplacian,glos2018vertices}, we claim that there is a need to
design the entropy formula which accepts each of the above-mentioned matrices.

Entropy in the work \cite{de2016spectral} is defined as the von Neumann entropy 
of Gibbs state of Laplacian matrix
\begin{equation}
S(\varrho_L^\tau) = - \tr \left( \frac{\exp(-\tau L)}{Z} \log  \frac{\exp(-\tau 
L)}{Z} \right),
\end{equation}
where $Z$ is a normalizing constant. Formal introduction of this concept will be
presented in the Preliminaries. Numerical calculations shed light on interesting
behavior of the entropy depending on the parameter $\tau$ of the Gibbs state
interpreted as a parameter proportional to the inverse of temperature \cite{nicolini2018thermodynamics,de2016spectral} or evolution time \cite{nicolini2018thermodynamics,gibbs1902elementary,ghavasieh2020enhancing,de2016spectral}. The authors of
\cite{de2016spectral}  point the phase transition of entropy value for \ER and
\WS graphs for some critical value $\tau_{\rm crit}$. Our analytical
considerations on \ER graphs confirm that such a phase transition actually
occurs, however,  the value of $\tau_{\rm crit}$ depends on the graph order.

Depending on a graph, the phase transitions occurs either for smaller or larger
values of $\tau$. The direction of phase transition change may be derived from
the analysis of entropy limits of graphs with increasing graph order: when the
entropy for fixed $\tau$ grows like $\log(n)$, then clearly the phase transition
moves to the right. On the other hand, when entropy converges to zero, then the phase
transition moves to the left.

For this reason, we calculated the entropy for some special graph classes for
fixed parameter $\tau$ and changing graph order $n$. We made the entropy
analysis for a few types of graph matrices, that is adjacency matrix, Laplacian
and normalized Laplacian. It appeared that the entropy usually takes the form
either $o(1)$ or $\log n- O(1)$, which shows that the phase transition  moves
respectively to the left or right. Furthermore, the deviations from $\log n$
differ for different random graph models, which can give a clue about their
properties. On top of that, we made a numerical analysis for the entropy of a
few real-world graphs analyzing the location and the shape of its phase
transition.

This work is organized as follows. We begin with preliminaries in
Section~\ref{sec:preliminaries}. Then, in Section~\ref{sec_general_theorems} we
present general theorems for entropy behavior basing on properties of the
matrix spectra. The entropy values for specific graph classes are presented in
Section~\ref{sec_specific_graph_classes}. The entropy behavior studied
for various random graph models and real-world graphs is described in
Section~\ref{sec:random_graphs}. Eventually, conclusions can be found in
Section~\ref{sec:conclusions}.

\section{Preliminaries}\label{sec:preliminaries}
We will be interested in studying the von Neumann entropy of Gibbs states
associated with a graph $G$. A graph $G$ is a pair $(V,E)$ where $V$ is a set of
vertices and $E$ is a set of edges. In this work we restrict ourselves to
simple undirected graphs. A graph has three typical matrix representations: the
adjacency matrix, the Laplacian matrix and the normalized Laplacian matrix. The
adjacency matrix of a simple graph is a symmetric square matrix consisting of
ones if two vertices are adjacent and zeros otherwise. The adjacency matrix of a
graph $G$ will be denoted $A(G)$. The degree matrix is a diagonal matrix with
degrees of vertices on the diagonal. The degree matrix will be denoted $D(G)$.
We will often make use of (combinatorial) Laplacian matrix which is defined as
$L(G) \coloneqq D(G)-A(G)$. The normalized Laplacian is defined as
$\mathcal{L}(G) \coloneqq D(G)^{-\frac{1}{2}} L(G) D(G)^{-\frac{1}{2}} = \Id -
D(G)^{-\frac{1}{2}} A(G) D(G)^{-\frac{1}{2}}$. When it will not make confusion
we will be writing only $\mathcal{L}$ instead of $\mathcal{L}(G)$ and
analogously for other graph matrices. 
Eigenvalues of matrices will be denoted $\lambda_{1} , \ldots,
\lambda_n$, where $\lambda_{1} \geq \cdots \geq \lambda_n$. 

In this paper we will use the big-O notation. Class $O(f(n))$ denotes a set of
functions $g$ such that there exist $c>0$ and $n_0\in \Z_{>0}$ s.t. for all
$n\geq n_0$ we have $|g(n)|\leq cf(n)$. We write $f(n)=\Theta(g(n))$ iff
$f(n)=O(g(n))$ and $g(n)=O(f(n))$. Finally, class $o(f(n))$ denotes set of
functions $g$ s.t. $\lim_{n\to\infty} g(n)/f(n)=0$. In particular $O(1)$ denotes
a set of functions upperbounded in absolute value by a constant, and $o(1)$
denotes a set of functions converging to 0.

Now we will introduce the von Neumann entropy of a quantum state $\varrho$. As
$\varrho$ is a density matrix, it is positive and has unit trace, its
eigenvalues form a probability vector. Thus, the von Neumann entropy of the
state $\varrho$ is defined as the standard Shannon entropy of its eigenvalues.
This fact can be succinctly written as
\begin{equation}
S(\varrho) = -\tr (\varrho \log(\varrho))
\end{equation}
where $\log$ refers to the natural logarithm throughout this paper.

For any Hermitian operator $H$ we can define an associated Gibbs state
$\varrho_H^\tau$ as
\begin{equation}
\varrho_H^\tau = \frac{\exp(-\tau H)}{Z},
\end{equation}
where $Z = \Tr(\exp(-\tau H))$ is the partition
function~\cite{ghavasieh2020enhancing}. The parameter $\tau$ can be regarded
either as a parameter proportional to the inverse of the temperature \cite{nicolini2018thermodynamics,de2016spectral} or the
diffusion time~\cite{nicolini2018thermodynamics,gibbs1902elementary,ghavasieh2020enhancing,de2016spectral}. %
Note that the von Neumann entropy of the Gibbs state can be written as
\cite{de2016spectral}
\begin{equation}
S(\varrho_H^\tau)
= \tau \Tr \left(H \varrho_H^\tau \right) + \log Z.
\end{equation}

This entropy has two simple properties summarized in the following lemma, which 
proof is stated in the Supplementary Materials in Section~1.%
\begin{lemma}\label{lm:properties_with_tau}
Let $H$ be a positive semidefinite matrix and $c \in \R$. It holds that
$S\left( \varrho_{cH}^\tau \right)
= S\left( \varrho_H^{c\tau} \right)$
and 
$S\left( \varrho_{c\Id + H}^\tau \right)
= S\left( \varrho_H^{\tau} \right).$
\end{lemma}

We will be  writing $S\left( \varrho_H \right)$ instead of $S\left(
\varrho_H^\tau \right)$ when the value $\tau$ does not need to be stated
explicitly. 

When calculating the entropy of a graph given by the adjacency matrix 
we will use the notation  $S(\varrho_{A})$ for $S(\varrho_{-A})$.
When dealing with the Laplacian and normalized Laplacian matrices we will be 
writing  $S(\varrho_{L})$ and $S(\varrho_{\mathcal{L}})$ respectively.

Finally, let us present a simple proposition describing the limit behavior of 
graph entropy.
\begin{proposition}
Assume $G$ is be a connected graph of order $n$. Then, for $M\in \{A,L,\mathcal 
L\}$ we have
\begin{gather}
S(\varrho^0_{M(G)}) = \log(n),\\
\lim_{\tau \to + \infty} S(\varrho^\tau_{M(G)}) = 0.
\end{gather}
\end{proposition}
The proof can be found in the Supplementary Materials in Section~2. 
In fact, the proof 
shows that even for not connected graphs 
the entropy converges to $\log(n)$ as $\tau \to 0$. 
On the other hand, for $\tau \to\infty$ for Laplacian and 
normalized Laplacian the entropy converges to $\log(k)$, where $k$ is the 
number of connected components of $G$. For adjacency matrix the limit for 
non-connected graphs may depend on the form of connected components. Note that 
by the proposition for connected graph the entropy continuously changes from 
$\log(n)$ to zero, when $\tau$ changes from zero to infinity.

\section{General entropy properties}\label{sec_general_theorems}

In this section we will present general theorems concerning the 
entropy's behavior in which we assume only some restrictions on matrix spectra.

Let us begin with a proposition which shows a useful property of $d$-regular
graphs. A $d$-regular graph is a graph whose all vertices have degree equal to
$d$.
For continuous-time quantum walk on $d$-regular graphs the evolution is 
independent on the choice of 
either adjacency matrix or Laplacian \cite{childs2004spatial}.
It follows from the fact that $D=d\1$ and hence it affects only the global 
phase.
For a similar reason, in the case of normalized Laplacian it can be seen as a
change of time.

It turns out that the proposed  entropy reflects this behavior.
\begin{proposition}\label{prop:regular_graph}
Let $G$ be a $d$-regular graph. Then 
$S(\varrho^\tau_{A}) = S(\varrho^\tau_{L})$
and $S(\varrho^\tau_{\mathcal{L}}) = S(\varrho^{\tau/d}_{A})$. 
\end{proposition}

\begin{proof}
Let $G$ be a $d$-regular graph. 
Then Laplace matrix of $G$ is 
$L = d \Id - A$, where $A$ is the adjacency matrix of $G$. 
Now from Lemma \ref{lm:properties_with_tau} we have that 
$S(\varrho^\tau_{L})= S(\varrho^\tau_{d \Id - A}) = S(\varrho^\tau_{A})$.

The normalized Laplacian for the $d$-regular graph takes the form 
$\mathcal{L} = \1 -\frac{1}{d}A$. Therefore again from Lemma 
\ref{lm:properties_with_tau} we have 
\begin{equation}
S(\varrho^\tau_\mathcal{L})=S(\varrho^\tau_{\1 -\frac{1}{d}A}) = 
S(\varrho^\tau_{-\frac{1}{d}A}) = S(\varrho^{\tau/d}_{A}).
\end{equation}
\end{proof}

It turns out that for the normalized Laplacian the entropy may take the values 
only from the very small interval.
Let us first present a result for general Hermitian matrices with bounded 
spectra.
Its proof can be found in the Supplementary Materials in Section~3.1. %

\begin{lemma} \label{lm:finite-spectrum}
Let $ H$ be a matrix with eigenvalues bounded by $c_1 \geq \lambda_i 
\geq c_2$.  Let $\tau>0$ be a constant. Then
\begin{itemize}
\item if $c_1,c_2\leq 1/\tau$, then 
\begin{equation}
 \log n - S(\varrho_H)  \leq \tau (c_1 - c_2),
\end{equation}
\item if $c_2\leq 1/\tau \leq c_1$, then 
\begin{equation}
\log n - S(\varrho_H) \leq 		
\tau \left(c_1- \min \{ c_1 \exp (\tau (c_2 - c_1)) , c_2 \}  \right),
\end{equation}
\item if $c_1,c_2\geq 1/\tau$, then 
\begin{equation}
\log n - S(\varrho_H) \leq  \tau c_1 \left( 1- \exp \left(\tau (c_2 - 
c_1)\right)  \right).
\end{equation}
\end{itemize}	
\end{lemma}

Conclusion directly drawn from the above Lemma is stated as a theorem 
concerning the entropy of a sequence of positive semidefinite matrices with 
finite spectral norm.

\begin{theorem}\label{theorem:finite-spectrum}
Suppose $(H_n)$ is a sequence of positive semidefinite matrices $n\times n$
with spectral norm
bounded by some constant independent of $n$. Then for fixed $\tau$ it holds that
$S(\varrho_{H_n}^\tau)=\log n- O(1)$.
\end{theorem}

For normalized Laplacian we have 
$c_2=0$ and $c_1=\|\mathcal{L}\|\leq 2$ 
\cite{chung1997spectral}, which give us the situation as in Lemma 
\ref{lm:finite-spectrum}.
More specifically, independently on $\| \mathcal{L}\|$ and $\tau$ the bound 
yields 
\begin{equation}
\log n - S(\varrho_{H}) \leq \tau \| \mathcal{L}\|.
\end{equation}
The bound cannot be improved to $\log n-o(1)$ for general normalized Laplacians
sequence of increasing size. In fact, we will show that the deviations from 
$\log(n)$ occur not only for simple graphs like cycle, but also for all 
complex graphs considered in this paper, see Sec.~\ref{sec:random_graphs}.

Note that for Laplacian matrices of graphs with maximal degree $\Delta$ we have 
$\Delta \leq \|L\| \leq 2\Delta$ \cite{anderson1985eigenvalues}. Furthermore, for arbitrary graph we have 
$c_2=0$ for the Laplacian. Hence if a graph has a bounded degree, then 
we can simply utilize Theorem \ref{theorem:finite-spectrum} in this 
scenario.

While considering Laplacian matrices we need to assume that a matrix is
singular. More specifically, the number of zero eigenvalues is equal to the
number of connected components of the graph. We will
focus on the case when one of the eigenvalues is equal to zero and the rest
of the eigenvalues are strictly positive (\ie \ the graph is connected). In the 
next theorem we restrict
ourselves to the case when all the nonzero eigenvalues converge to a positive
constant.

\begin{theorem} \label{theorem:normalized-laplacian}
Let $H$ be a singular nonnegative matrix of size $n$ with single 
zero-eigenvalue and let $\tau>0$ be a constant. Assume that 
$\lambda_1\to c$ and $\lambda_{n-1}\to c$ for some constant $c$ as $n 
\rightarrow \infty$. 
Then  $S(\varrho_{H})=\log n -o(1).$
\end{theorem}

The proof of the above theorem can be found in the Supplementary Materials in  
Section~3.2.

Now we focus on the case when the spectrum can be unbounded. 
An example of such a matrix is the Laplacian matrix. 
While it is singular and positive semidefinite, its 
norm coincides with the maximum degree of the graph, hence it can be unbounded. 
In the following theorem, proven in the Supplementary Materials in  
Section~3.3,  we make an assumption only on the 
behavior of the smallest nonzero eigenvalue.

\begin{theorem}\label{theorem:laplacian-like-entropy1}
Let $H_n$ be a singular nonnegative matrix of size $n$ with single
zero-eigenvalue  and let $\tau>0$ be a constant. Assume $\lambda_{n-1}(H_n)
\gg\log n$. Then $S\left (\varrho_{H_n} \right)=o(1)$.
\end{theorem}
We use the notation $f(x) \gg g(x)$ when 
$\lim\limits_{x \rightarrow \infty}f(x)/ g(x) =\infty$.

The Laplacian matrix of a connected graph does not necessarily satisfy the
assumption on $\lambda_{n-1}$ mentioned in
Theorem~\ref{theorem:laplacian-like-entropy1}, hence the result cannot be
generalized into `arbitrary sequence of Laplacians', even connected. As an
example, the cycle graph $C_n$ of size $n$ is known to have eigenvalues
$2-2\cos(\frac{2\pi j}{n})$ for $j=0,\dots,n-1$ \cite{brouwer2011spectra}. 
Hence the spectrum is bounded and we can apply 
Theorem~\ref{theorem:finite-spectrum}.
By this we have $S(\varrho_{L(C_n)}) = \log n -O(1)$. Such behavior shows the 
difference between Laplacian and normalized Laplacian in the sense of von 
Neumann entropy of the Gibbs state. 

\section{Entropy of specific graph classes}\label{sec_specific_graph_classes}

In this section we study the entropy of a few selected classes of graphs. 
The entropy is calculated for three types of graph matrices: adjacency matrix 
$A$, Laplacian matrix $L$ and normalized Laplacian $\mathcal{L}$. 
Four types of graphs were taken into consideration: empty graph, complete 
graph, bipartite graphs and cycle graph.
An empty graph of order $n$ is denoted by $E_n$.
The symbol $K_n$ denotes the complete graph. 
A bipartite graph is a graph whose vertices are partitioned into two
disjoint sets, $V$ and $W$,
and any two vertices from the same set cannot be adjacent.
When a vertex $v \in V$ is adjacent to all vertices
from the set $W$ and vice-versa, then the graph is called a complete bipartite graph.
Such a complete bipartite graph, where $|V|=n_1$ and $|W|=n_2$, is denoted by 
$K_{n_1,n_2}$.
Finally, the symbol $C_n$ is used to denote a cycle graph.

\begin{table}[h!]
\begin{center}
\begin{tabular}{ c|c c c } 
& adjacency & \multirow{2}*{Laplacian} & normalized \\
& matrix & & Laplacian \\
 \hline 
 $E_n$ & $\log n-o(1)$ & $\log n-o(1)$ & --  \\ 
 $K_n$  & $o(1)$ & $o(1)$ & $\log n-o(1)$ \\[1ex]
 $K_{n_1,n_2}$ & $o(1)$ & \multicolumn{2}{c}{depends on $n_1,n_2$}  \\
 $K_{n_1,n_1}$ & $o(1)$ & $o(1)$ & $\log n-o(1)$  \\ 
 $K_{n_1,1}$ & $o(1)$ & $\log n-o(1)$ & $\log n-o(1)$ \\[1ex] 
  $C_n$ & $\log n-\Theta(1)$ & $\log n-\Theta(1)$ & $\log n-\Theta(1)$\\
\end{tabular}
\end{center}
\caption{\label{table:entropy_of_graph_classes} Asymptotic behavior of the 
entropy calculated for various graph classes described in 
Sec.~\ref{sec_specific_graph_classes} 
}
\end{table}

All the results are presented in Table \ref{table:entropy_of_graph_classes}.
The proofs can be found in the Supplementary Materials in  Section~4. 
An interesting observation is that in the first three cases the entropy behaves 
either like $\log n$ or converges to zero. 
For a cycle graph however the result is neither of them.
More specifically, the entropy calculated for both 
adjacency and Laplacian matrices behaves in the same way
\begin{equation}\label{eq:entropy_cycle_adj}
S(\varrho_{A(C_n)}) = S(\varrho_{L(C_n)})= \log n - 2 \tau \frac{I_1 
(2\tau)}{I_0 (2 
\tau)} +\log\left(I_0 (2\tau)\right) + o(1),
\end{equation}
where $I_\alpha(x)$ is the modified Bessel function of the first kind.
For the normalized Laplacian of a cycle we obtain
\begin{equation}\label{eq:entropy_cycle_normalized_lapl}
S(\varrho_{\mathcal{L}(C_n)}) = \log n -  \tau \frac{I_1 
(\tau)}{I_0 ( \tau)} +\log\left(I_0 (\tau)\right) + o(1).
\end{equation}

It is also worth noting that the entropies calculated for adjacency matrix and 
Laplacian usually have the same asymptotic properties, that is either $\log 
n-o(1)$ or $o(1)$. 
Nevertheless, we found an counterexample which is a star graph $K_{n_1,1}$ for 
which the entropy for adjacency matrix is substantially different than the 
entropy for Laplacian.

\section{Random graphs}\label{sec:random_graphs}
In this section we consider various random graph models. Let us begin with \ER
random graphs \cite{erdHos1960evolution}. The symbol $G(n,p)$ is used to denote
a random graph of order $n$ where the probability that any two vertices are
adjacent equals $p$. A generalization of the \ER graph model is the \CL graph
model \cite{chung2004spectra,chung2011spectra} in which we obtain a graph with a
specified expected degree sequence $(w_1, \ldots, w_n)$. The probability that
vertices $v_i$ and $v_j$ are adjacent equals $w_i w_j/\sum_k w_k$.

\WS random graphs \cite{watts1998collective} are constructed as follows. In the
first step we have a regular ring lattice, that is a graph of order $n$ where
each vertex is adjacent to $K$ neighbors ($K/2$ on each side). Then, for each
vertex we consider their neighbors from one side and rewire them with
probability $\beta$ to some other vertex. \WS graphs are known to be
small-world, meaning that in contrary to \ER graphs all vertices are close to
each other. Nevertheless, the degree distribution is highly concentrated around
$K$.

\BA random graphs \cite{albert2002statistical} are constructed as follows. We begin with 
a complete graph with fixed order $m_0$.  
Then we add vertices one after another. 
Each time, a new vertex is adjacent to $m$ of the already existing 
vertices.
The probability that the new vertex is adjacent to the already-existing vertex 
$v$ is proportional to the degree of the vertex $v$. 

We will start with analytical results for 
\ER and \CL graphs for Laplacian and  normalized Laplacian matrices.
Then, we will present numerical results for other types of graph matrices and 
other graph models presented above.

\subsection{\ER graphs}

The Laplacian matrix of a random \ER graph with 
$p \gg\log(n)/n$ almost surely has a single outlying zero eigenvalue 
and  the rest of eigenvalues  behaving like $np(1+o(1))$. 
A useful property of the second smallest eigenvalue is formulated as a theorem.
\begin{theorem}[\cite{kolokolnikov_algebraic_2014}]
The second smallest eigenvalue $\lambda_{n-1}$ of the random Laplacian 
matrix $L$ from Erd\H{o}s-R\'enyi graph $G(n,p)$ with $p \gg\log 
(n)/n$ satisfies a.a.s.
\begin{equation}
\lambda_{n-1}=np +O(\sqrt{np\log n}).
\end{equation}
\end{theorem}
Moreover, from \cite{glos2018vertices} we have that 
$\lambda_1 \sim np$ for $p\gg\log(n)/n$.
The next remark follows from Theorem \ref{theorem:laplacian-like-entropy1}.
\begin{remark}\label{remark:erdos_renyi_dense}
The von Neumann entropy of Gibbs state of Laplacian of random 
\ER graph $G(n,p)$ with $p \gg\log(n)/n$ converges 
a.a.s. to zero.
\end{remark}
The main reason of such behavior is the strongly outlying 0 value. The 
behavior changes when $p=\Theta(\log(n)/n)$. For $p<(1-\varepsilon)\log (n)/n$ 
the graph is almost surely disconnected \cite{erdHos1960evolution}, and since 
the dimensionality of the null-space of the Laplacian equals the number of 
connected components \cite{brouwer2011spectra}, the graph entropy strongly 
depends on $n$.  

Let us now consider the threshold behavior of \ER model when
$p=p_0\frac{\log n}{n}$ with $p_0>1$. Here we have $\lambda_{n-1}\sim
(1-p_0)W_{-1}^{-1}\left(\frac{1-p_0}{\ee p_0}\right)\log n$ 
\cite{kolokolnikov_algebraic_2014} and
$\lambda_1\sim (1-p_0)W_{0}^{-1}\left(\frac{1-p_0}{\ee p_0}\right)\log n$
\cite{glos2018vertices}, where $W_{-1},W_0$ are Lambert $W$ functions. In this
case the following theorem provides results for selected values of $\tau$. Its
proof can be found in the Supplementary Materials in  Section~3.4.%

\begin{theorem}\label{th:th7}
Let $H_n$ be a positive semidefinite matrix with a single zero-eigenvalue of 
size $n$ and $\tau>0$ be a constant. Assume $\lambda_{n-1}=a\log n$ and 
$\lambda_1=b\log n$ for $a,b>0$. Then the behavior of the von Neumann 
entropy satisfies
\begin{enumerate}
\item if $\tau <\frac{1}{b}$, then $S(\varrho_{H_n})\geq (1-\tau b)\log n 
+o(1)$,
\item if $\tau =\frac{1}{b}$, then $S(\varrho_{H_n})\geq \log 2 +o(1)$,
\item if $\tau >\frac{1}{a}$, then $S(\varrho_{H_n})=o(1)$.
\end{enumerate}
\end{theorem}
For random \ER graphs the above theorem translates to the following remark.
\begin{remark}\label{remark:th7_for_ER}
Let $H_n$ be a Laplacian matrix of a random \ER graph for 
$p=p_0\frac{\log n}{n}$ with $p_0>1$.  Then 
\begin{enumerate}
\item if $\tau <W_{0}\left(\frac{1-p_0}{\ee p_0} 
\right)/(1-p_0)$, 
then a.a.s. $S(\varrho_{H_n})\geq C\log n +o(1)$
for some $C \in (0,1)$.
\item if $\tau > W_{-1}\left(\frac{1-p_0}{\ee p_0} \right)/(1-p_0)$, then 
a.a.s. 
$S(\varrho_{H_n})=o(1)$.
\end{enumerate}
\end{remark}

Theorem \ref{th:th7} and Remarks \ref{remark:erdos_renyi_dense}, 
\ref{remark:th7_for_ER} 
give an analytical justification for the effect presented
in~\cite{de2016spectral}. The authors pointed that the phase-transition occurs
with changing $\tau$. This phase transition is shown in
Figure~\ref{fig:ph-trans}, which shows the value of the entropy of the Gibbs
state for an \ER graph with a function of the dimension of the graph and
the parameter $\tau$. We show three values of the parameter $p_0$, namely
$p_0=10.5, \; 21, \; 42$. To make it easier to compare the values for changing
dimensionality, the value of the entropy is normalized by dividing by $\log n$.
The phase transition is clearly visible. We should also note that for
sufficiently large dimension $n$ the normalized entropy does not depend on the
dimension $n$ around $\tau < \frac{1}{b}$. Yet, it still depends on $\tau$ as
stated by Theorem~\ref{th:th7}. A more detailed view on this phenomenon is
presented in Figure~\ref{fig:phase-transition}. It depicts this phase transition
for the ER, WS and BA models and for all considered graph matrices. The model
specific parameters are stated in the legend.
\begin{figure}[!htp]
	\centering\includegraphics{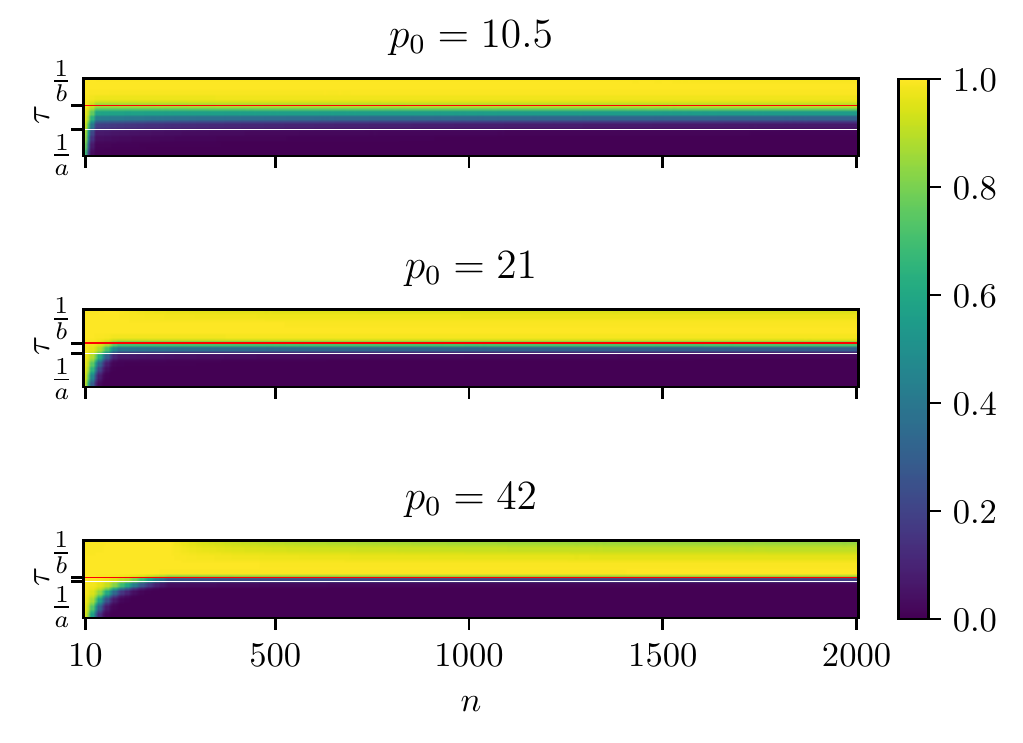}
	\caption{Entropy of the Gibbs state as a function of the parameter
	$\tau$ and the dimension of the graph, $n$ for the \ER model. The value of
	the entropy is normalized by dividing by $\log n$. The phase transition can
	be easily seen. We show results for three values of the parameter $p_0$.
	The horizontal lines mark the theoretical boundaries for $\tau$ found in
	Theorem~\ref{th:th7} and Remark~\ref{remark:th7_for_ER}. The red line marks $\tau=\frac{1}{b}$ while the white one corresponds to $\tau=\frac{1}{a}$.}
	\label{fig:ph-trans}
\end{figure}

Theorem~\ref{th:th7} not only confirms that there is a strong correlation
between spectral gap and the critical value of $\tau$ but also shows that the
transition depends on the order of the graph $n$. Further numerical
investigation shows that the entropy stabilizes with the graph order.

Let us now focus on the normalized Laplacian. It is known that normalized
Laplacian of random \ER graph satisfies  requirements of Theorem
\ref{theorem:normalized-laplacian} for $p\gg \log(n)/n$ \cite{chung2011spectra},
however, we can go beyond that. The assumption can be relaxed to
$pn=(1+\varepsilon)\log n$ for $\varepsilon>0$ by Corollary 1.2 from
\cite{kolokolnikov_algebraic_2014}. We conclude our results with the following
remark.
\begin{remark}
Assume $\mathcal L$ is a normalized Laplacian matrix of random 
\ER graph with $p\geq(1+\varepsilon)\log n/n$. The von Neumann 
entropy of Gibbs state 
satisfies
$S(\varrho_{\mathcal L}) - \log(n) 
\underset{\text{a.a.s.}}{\longrightarrow}0.$
\end{remark}

\subsection{Chung-Lu graphs}
By Theorem~4 from \cite{chung2011spectra}, normalized Laplacian of a random 
Chung-Lu graph  for which minimum expected degree 
$\omega_{\mathrm{min}} \gg\log n$ 
satisfies the requirement of Theorem~\ref{theorem:normalized-laplacian}.
Therefore we have the following remark.
\begin{remark}
Assume $\mathcal L$ is a normalized Laplacian matrix of a  Chung-Lu random 
graph for which minimum expected degree satisfies  
$\omega_{\mathrm{min}} \gg \log n$. The von Neumann entropy of 
Gibbs state 
satisfies
$S(\varrho_{\mathcal L}) - \log(n) 
\underset{\text{a.a.s.}}{\longrightarrow}0.$
\end{remark}

The following remark concerns the case of adjacency matrix of a Chung-Lu 
random graph. 
Its proof can be found in the Supplementary Materials in Section~3.5.%

\begin{remark}\label{remark:chung_lu_adjacency}
Let $A$ be an adjacency matrix of a random Chung-Lu graph with the maximum 
expected 
degree satisfying $\omega_{\mathrm{max}}> \frac{8}{9}\log(\sqrt{2}n)$
and 
$\tilde{d} := \frac{\sum \omega_i^2}{\sum \omega_i} \gg 
\omega_{\mathrm{max}}\sqrt{\log n} $.
Then  $S(\varrho_A) = o(1)$.
\end{remark}

\subsection{Numerical insight}\label{sec:numerical} In this section we will
complement the analytical results from previous sections by numerical studies on
various random graphs as well as some real-world graphs. %
Basing on the results in \cite{de2016spectral} we expect that the information
whether the graph describes real-world interactions may be distilled from the
location and shape of the phase-transition.

We can clearly observe that the entropy function in $\tau$ differs among \ER
graphs and \WS networks. Nevertheless, in the case of \ER and \BA graphs we
observe a similar shape of the plots with a single inflection point, however
there is a difference in location. Furthermore, in
Figure~\ref{fig:phase-transition} we also presented the shape of the curve for
smaller graphs. We can see that for all values of $p$, the location of phase
transition for \ER graphs goes to larger values of $\tau$, which is contrary to
\WS and \BA.

\begin{figure}[!htp]
\centering\includegraphics{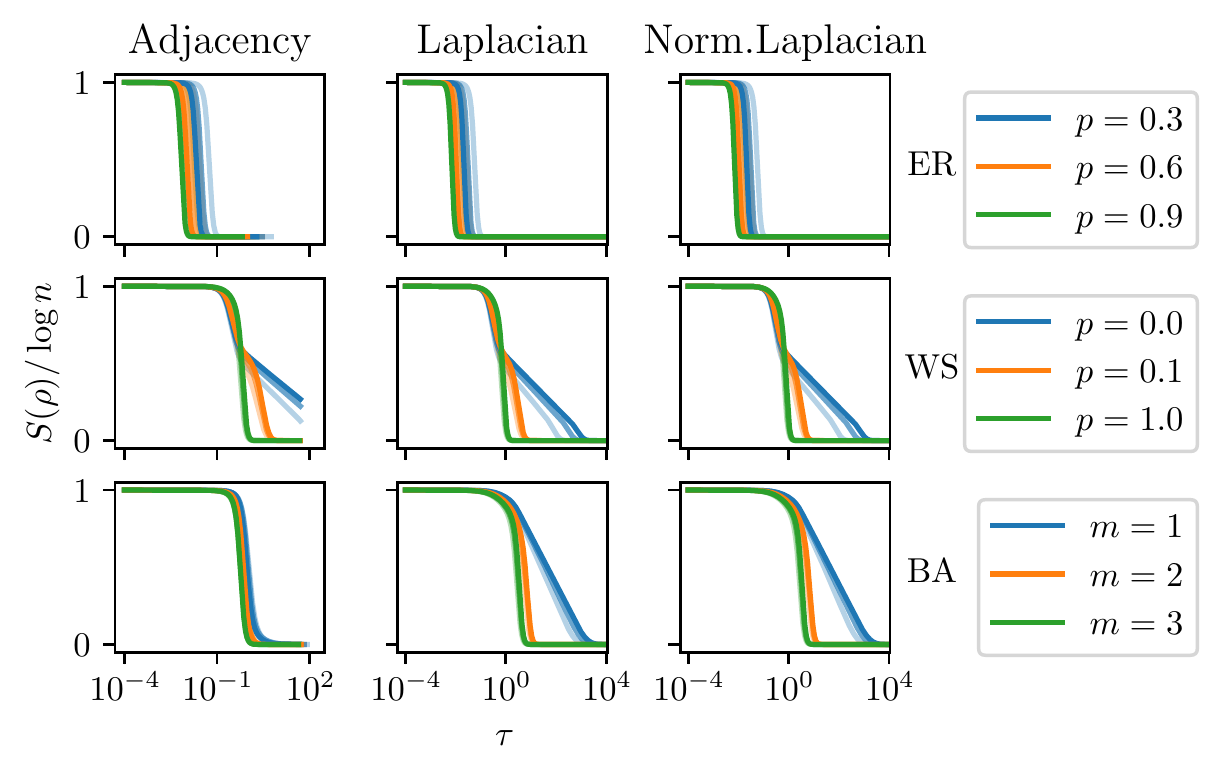} \caption{Illustration of the entropy's
phase transition for the ER, WS and BA models and all considered graph matrices.
The value of the entropy is normalized by dividing by $\log n$. The dimension is
$n=1200$. For WS we choose parameters $K=4$ and $\beta=0.6$. The specific model
parameters are stated in the corresponding legends. Each plot is obtained by
averaging 100 randomly sampled graphs. The transparency in the plot denotes the
size of the subgraph of the originally sampled random graph. In increasing 
order of opacity, these lines correspond to 33\%, 66\% and 100\% of the total 
number of nodes of the sampled graph chosen for 
calculations. Subgraphs were generated
by choosing a vertex with the highest degree, and  33\% and 66\% vertices
nearest to it.}\label{fig:phase-transition}
\end{figure}

We expect to observe similar situation for real-world graphs. More specifically,
we focused on co-authorship graphs (HEP-PH, HEP-TH, CA)
\cite{web_hepph,web_hepth,web_ca}, social networks (Facebook FB, Twitch TW)
\cite{web_fb,web_tw}, Gnutella graph (GT) \cite{web_gt} and as-caida (CAIDA)
\cite{web_caida} graphs. All the plots are presented in the Figure
\ref{fig:real_graphs}. Moreover, for the sake of comparison we considered \ER
graphs chosen so that the number of vertices was the same as in the
corresponding real-world graph and the expected number of edges equals the
number of edges of the real-world graph. Finally, we also calculated the entropy
of subgraphs of real graphs to analyze how the phase transition changes with the
graph size.

For some graphs we observe nontrivial changes in the pace of entropy change,
similarly to as it was in \WS graphs (see Figure \ref{fig:phase-transition}).
This is the most prominent in the case of Facebook for Laplacian and normalized
Laplacian, but for these matrices a similar effect can be observed also for
HEP-PH, HEP-TH and GT. It is worth noting that these pace changes occur
independently on the type of graph. More precisely, for co-authorship graphs the
pace changes are clearly visible for HEP-PH and HEP-TH, while they are not
visible for CA. This is even more appealing in the case of social network graph,
that is the pace changes are very clear for FB graph while they are not visible
for other graph.

\begin{figure}[!htp]
\centering\includegraphics{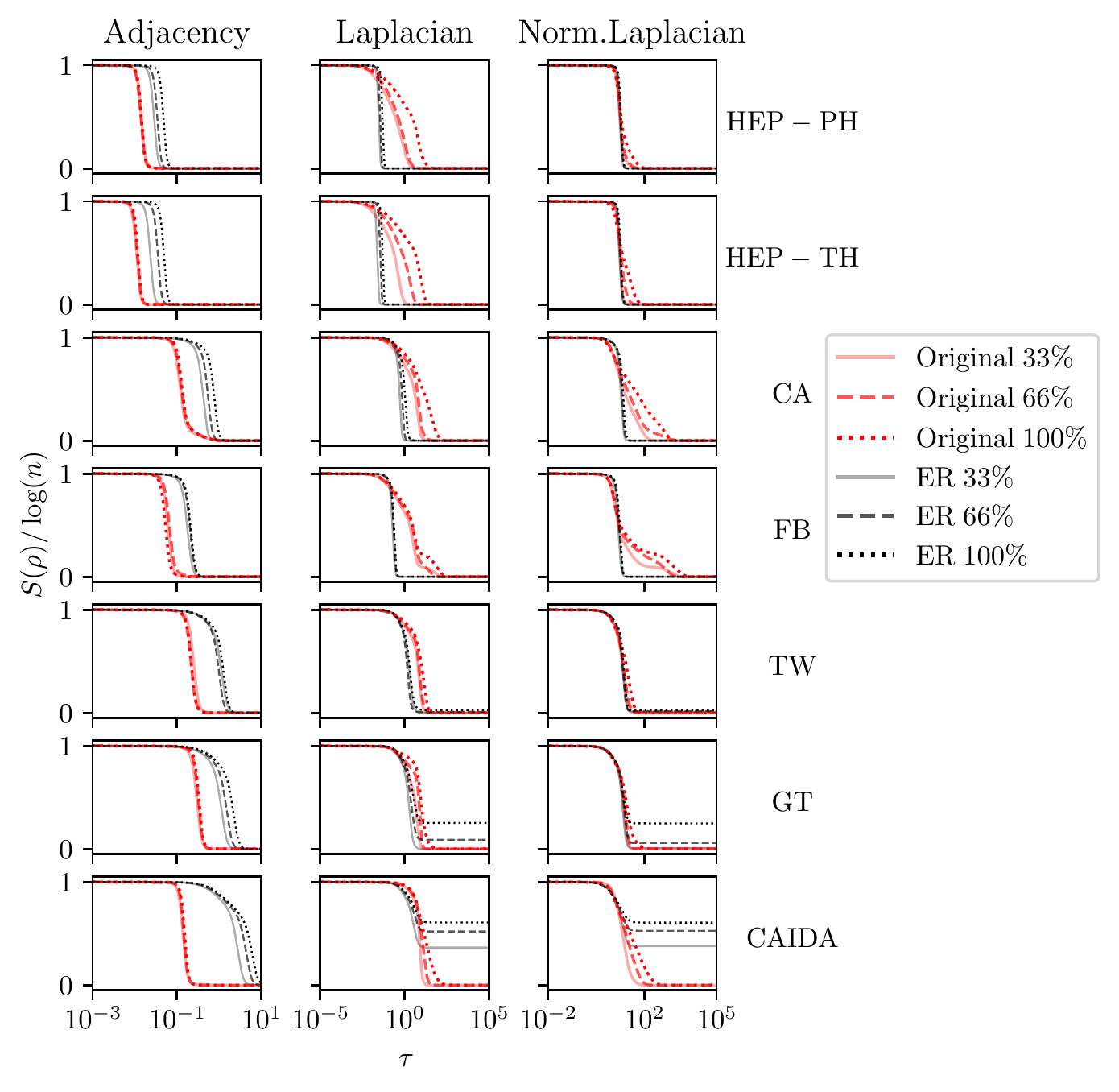} 

\caption{Illustration of the entropy's phase transition for real-world graphs.
	Each graph was turned into a simple undirected graph by replacing directed
	edges with undirected edges. Then, for each graph we chose the largest
	connected component. For each real-world graph, we generated ten \ER graphs
	with parameter $p=2 m/n^2$ where $n,m$ are the numbers of vertices and edges
	of the largest connected components of real-world graphs respectively.
	Finally, we took an average entropy. Subgraphs of real graphs were generated
	by choosing a vertex with the highest degree, and  33\% and 66\% vertices
	nearest to it.}\label{fig:real_graphs}
\end{figure}

For all real-world graphs for adjacency matrix, the phase transition occurs for
larger  values of $\tau$ than for the corresponding \ER graphs. Contrary to
adjacency matrix, for Laplacians and normalized Laplacians the phase transition
starts roughly at the same value of $\tau$ for both real-world and \ER graphs. 
On the other hand, phase transitions usually are more rapid for random graphs.
Different behavior can be observed for \ER graphs corresponding to GT and CAIDA
graphs. In those cases \ER graphs have many disconnected components and
therefore the limit as $\tau \rightarrow \infty$ is no longer zero.

Finally, there is almost no change in the shape and location of the phase transition of the entropy for
real graphs for adjacency matrix. In contrary, for corresponding \ER graphs we
observe that with the increasing number of nodes the location of the phase
transition moves to higher values of $\tau$. In the case of the Laplacian
matrix, we observe that the location of phase transition remains the same for
\ER graphs, while for real graphs it clearly goes to larger values of $\tau$.
Similar behavior is observed for the normalized Laplacian, however for some real
graphs (HEP-PH, HEP-TH) it is less evident compared to the Laplacian. The only
case for which the values of entropy was similar to the corresponding \ER graph
is the entropy of real graph TW for normalized Laplacian. Finally, the
non-trivial shape of the phase transition observed for FB can be found also for
subgraphs of FB, however for HEP-PH and HEP-TH it is observed only for the
original graph.

All the code used to obtain the results presented here is available on GitHub at
\url{https://github.com/iitis/graph-entropy}.

\section{Conclusions}\label{sec:conclusions} This work is focused on studying
the entropy of the Gibbs state for various graphs. We made the analysis for
three types of graph matrices: adjacency matrix, Laplacian and normalized
Laplacian for various graph classes. It turns out that the asymptotic properties
of the same graph may differ depending on which graph matrix is taken into
consideration.  We proved a few general theorems which assume only some 
constraints on matrix spectra.  Moreover, we studied several graph classes 
like complete graphs, bipartite graphs and cycle graphs, and derived the 
formulas for their entropy. It turned out that entropy usually takes the values 
either $\log n -o(1)$ or $o(1)$, which implies the shift of the phase 
transition.

We considered also various random graph models and real-world graphs. We focused
on the phase transition in $\tau$ of the entropy of \ER, \CL, \WS, \BA random
graphs with fixed graph order and some real-world graphs from various domains
like co-authorship and social networks. Analysis of real graphs shows that we can
indeed distill the information whether the graph represents some real-world
interactions. This can be distilled from the position and, in some
cases, the shape of the plot. The exact nature of this shift is dependent on
the chosen graph matrix, however for adjacency matrix and Laplacian the
difference were most evident.

\section*{Acknowledgments} AG has been partially supported by National Science
Center under grant agreement 2019/32/T/ST6/00158. AG would also like to
acknowledge START scholarship from the Foundation for Polish Science. AK and
{\L}P acknowledge the support of the Polish National Science Centre under the
project number 2016/22/E/ST6/00062.

\bibliographystyle{ieeetr}
\bibliography{random-graph-entropy}

\appendix
\section{Proof of properties of von Neumann entropy of the Gibbs
	state}\label{app:remark-properties}
Here we state the proof of claims made in Lemma~1.%
\begin{proof}
	\begin{equation}
	\begin{split}
	S(\varrho_{c H}^\tau)  
	&= -\tr \left (\frac{\exp(-\tau (cH))}{\tr\exp(-\tau (cH))} \log\left 
	(\frac{\exp(-\tau (cH))}{\tr\exp(-\tau (cH))}\right )\right ) \\
	&= -\tr \left (\frac{\exp(-(\tau c)H)}{\tr\exp(- (\tau c)H)} \log\left 
	(\frac{\exp(-(\tau c)H)}{\tr\exp(-(\tau c)H)}\right )\right ) 
	= S(\varrho_{H}^{c \tau}).   
	\end{split}
	\end{equation}
	\begin{equation}
	\begin{split}
	S(\varrho_{d\Id + H}^\tau)  
	&= -\tr \left (\frac{\exp(-\tau (d \Id + H))}{\tr\exp(-\tau (d \Id + 
		H))} \log\left (\frac{\exp(-\tau (d \Id + H))}{\tr\exp(-\tau (d \Id + 
		H))}\right )\right ) \\
	&= -\tr \left (\frac{\exp(-\tau d) \exp (H)}{\exp(- \tau d) \Tr 
		\exp(- \tau H)} 
	\log\left (\frac{\exp(-\tau d) \exp (H)}{\exp(- \tau d) \Tr 
		\exp(- \tau H)}\right )\right ) \\ 
	&= -\tr \left (\frac{\exp (H)}{\Tr \exp(- \tau H)} 
	\log\left (\frac{\exp (H)}{\Tr \exp(- \tau H)}\right )\right ) 
	= S(\varrho_{H}^\tau).   
	\end{split}
	\end{equation}
\end{proof}

\section{Entropy limits for $\tau \rightarrow 0$ and $\tau \rightarrow 
	\infty$.}\label{app:limits}
Assume we have a Hermitian matrix $M$.
\subsection{$\tau=0$}
\begin{equation}
\begin{split}
S(\varrho_{M}^0) & = -\tr \left(  \frac{\1}{n} \log \frac{\1}{n}\right) = \log n. 
\end{split}
\end{equation}

\subsection{$\tau\to\infty$}
Assume that 
$ \lambda_{1} \geq \lambda_{2} \geq \ldots \geq \lambda_{n-k} > 0$ and 
$\lambda_{n-k+1}, \ldots, \lambda_n = 0$ are eigenvalues of $M$. 
Defining $\exp_{i} \coloneqq \exp(-\tau \lambda_{i})$ and 
$\exp_{i,j} \coloneqq \exp(-\tau \left( \lambda_{i}-\lambda_{j}\right) )$
we have
\begin{equation}
\begin{split}
S(\varrho_{M}^\tau)  &= \tau \Tr (M \varrho^\tau_{M}) + \log Z \\
&=\frac{\tau \sum_{i=1}^n\lambda_i \exp(-\tau \lambda_i)
}{\sum_{i=1}^n\exp(-\tau \lambda_i)} +\log \left(\sum_{i=1}^n\exp(-\tau 
\lambda_i)\right) \\
&=\frac{\tau \lambda_1 \exp_1}{\sum_{i=1}^{n} \exp_i}
+ \ldots
+ \frac{\tau \lambda_{n-k} \exp_{n-k}}{\sum_{i=1}^{n} \exp_i}
+ \log \left(\sum_{i=1}^{n}\exp_i\right) \\
&= \frac{\tau \lambda_1 }{ 
	\sum_{i=1}^{n} \exp_{i,1}  }
+ \ldots 
+ \frac{\tau \lambda_{n-k} }{ 
	\sum_{i=1}^{n} \exp_{i,n-k}  } 
+ \log \left(k+ \sum_{i=1}^{n-k}\exp_i\right) \\
& \underset{\tau \rightarrow \infty}{\longrightarrow} \log k
\end{split}
\end{equation}
where the limit follows from observing that 
for the $j$-th factor the nominator grows like $\tau$, while in the 
denominator $\exp_{n,j}$ exponentially tends to infinity. More specifically
\begin{equation}
\lim\limits_{\tau \rightarrow \infty} \exp_{i,j} = 
\begin{cases}
0, \quad  i<j \\
1, \quad i=j   \\
\infty,  \quad  i>j.
\end{cases}
\end{equation}
Finally, it suffices to note that
$\lim\limits_{\tau \rightarrow \infty} \exp_{i} =0$  for every $i>0$.

Let us consider Laplacian and normalized Laplacian matrices. 
Since the number of 
eigenvalues equal to zero is equal to the number of connected components, then 
for $k=1$ we obtain the limit $\log(1)=0$. In the case of adjacency matrix, we 
can always shift the matrix by $\lambda_1(-A) \Id$ without the change of the 
entropy, see Lemma~1. 
Furthermore, for connected graphs there is a 
nonzero gap between the largest and second largest eigenvalue of the adjacency 
matrix which, similarly as in previous case, gives us $k=1$, and hence the 
limit is zero.

\section{Proofs of general theorems}

\subsection{Proof of Lemma~4 
}\label{app:lm_finite_spectrum_proof}
\begin{proof}
	The entropy takes the form
	\begin{equation}
	\begin{split}
	S(\varrho_H) 
	&= \tau \Tr (H \varrho^\tau_H) + \log Z \\
	&=\frac{\tau \sum_{i=1}^n\lambda_i \exp(-\tau \lambda_i)
	}{\sum_{i=1}^n\exp(-\tau \lambda_i)} +\log \left(\sum_{i=1}^n\exp(-\tau 
	\lambda_i)\right).
	\end{split}
	\end{equation}
	The numerator is a sum of eigenvalues mapped by $f(x) = \tau x \exp(-\tau x)$ 
	function. The function takes its unique maximum at $x=1/\tau$.
	
	Let us begin with the case when $c_1, c_2 \leq \frac{1}{\tau}$. Then
	\begin{equation}
	S(\varrho_H) 
	\geq \frac{n \tau c_2 \exp(-\tau c_2) }{n \exp(-\tau c_2)}
	+ \log \left( n \exp(-\tau c_1) \right)
	=\log n + \tau c_2 - \tau  c_1,
	\end{equation}
	and therefore
	\begin{equation}
	\log n - S(\varrho_H) \leq \tau (c_1 - c_2).
	\end{equation}
	If $c_1, c_2 \geq \frac{1}{\tau}$, then
	\begin{equation}
	\begin{split}
	S(\varrho_H) 
	&\geq \frac{n \tau c_1 \exp(-\tau c_1) }{n \exp(-\tau c_2)}
	+ \log \left( n \exp(-\tau c_1) \right) \\
	&=\tau c_1 \left( \exp \left( - \tau c_1 + \tau c_2\right) \right) + \log n - 
	\tau c_1 \\ 
	&= \log n + \tau c_1 \left( \exp \left(\tau (c_2 - c_1)\right) -1 \right),
	\end{split}
	\end{equation}
	and hence
	\begin{equation}
	\log n - S(\varrho_H) \leq \tau c_1 \left(1- \exp \left(\tau (c_2 - 
	c_1)\right) 
	\right).
	\end{equation}
	Assume finally that $c_2 \leq \frac{1}{\tau} \leq c_1$. In this case we have
	\begin{equation}
	\begin{split}
	S(\varrho_H) 
	&\geq \frac{n \tau \min \{ c_1 \exp(-\tau c_1), c_2 \exp(-\tau c_2) \}}{n 
		\exp(-\tau c_2)}
	+ \log \left( n \exp(-\tau c_1) \right) \\
	& = \tau \min \left\{ c_1 \frac{\exp (-\tau c_1)}{\exp (-\tau c_2)}, c_2 
	\frac{\exp (-\tau c_2)}{\exp (-\tau c_2)} \right\} + \log n - \tau c_1 \\
	& =\log n +  \tau \left(\min \{ c_1 \exp (\tau (c_2 - c_1)) , c_2 \}  - 
	c_1\right), 
	\end{split}
	\end{equation}
	and therefore
	\begin{equation}
	\log n- S(\varrho_H)   \leq  
	\tau \left(c_1- \min \{ c_1 \exp (\tau (c_2 - c_1)) , c_2 \}  \right).
	\end{equation}
\end{proof}

\subsection{Proof of Theorem~6 
}\label{app:th_normalized_laplacian_proof}
\begin{proof}
	The entropy takes the form 
	\begin{equation}
	\begin{split}
	S(\varrho_H) 
	&= \tau \Tr (H \varrho^\tau_H) + \log Z \\
	&=\frac{\tau \sum_{i=1}^n\lambda_i \exp(-\tau \lambda_i)
	}{\sum_{i=1}^n\exp(-\tau \lambda_i)} +\log \left(\sum_{i=1}^n\exp(-\tau 
	\lambda_i)\right).
	\end{split}
	\end{equation}
	Since the matrix $H$ is singular, we can extract a single zero eigenvalue. 
	Hence the first part of the sum can be bounded as
	\begin{equation}
	\frac{\tau (n-1) \lambda_{n-1} \exp(-\tau 
		\lambda_{1})}{1+(n-1)\exp(-\tau \lambda_{n-1})}
	\leq
	\tau \Tr (H \varrho^\tau_H)\leq 
	\frac{\tau (n-1) \lambda_1 \exp(-\tau 
		\lambda_{n-1})}{1+(n-1)\exp(-\tau \lambda_1)}
	\end{equation}
	Both bounds converge to $\tau c$ and hence $\tau \Tr (H \varrho^\tau_H)$ as 
	well converges to $\tau c$. 
	
	Similarly for $\log Z$ we have
	\begin{equation}
	\log(1+(n-1)\exp(-\tau\lambda_1))
	\leq	\log Z  \leq 
	\log(1+(n-1)\exp(-\tau\lambda_{n-1}))
	\end{equation}
	or equivalently
	\begin{equation}
	\log\left(\frac{1}{n}+\frac{n-1}{n}\exp(-\tau\lambda_1)\right)
	\leq \log Z -\log n\leq 
	\log\left(\frac{1}{n}+\frac{n-1}{n}\exp(-\tau\lambda_{n-1})\right)
	\end{equation}
	which implies $\log Z -\log n \to -\tau c$
	as $n\to\infty$, which finishes the proof.
\end{proof}

\subsection{Proof of Theorem~7 
}\label{app:th_laplacian_like_entropy_proof}

\begin{proof}
	The entropy takes the form 
	\begin{equation}
	\begin{split}
	S(\varrho_{H_n}) 
	&= \tau \Tr (H_n \varrho^\tau_{H_n}) + \log Z \\
	&=\frac{\tau \sum_{i=1}^n\lambda_i \exp(-\tau \lambda_i)
	}{\sum_{i=1}^n\exp(-\tau \lambda_i)} +\log \left(\sum_{i=1}^n\exp(-\tau 
	\lambda_i)\right).
	\end{split}
	\end{equation}
	Since $H_n$ matrix is 
	singular, we can extract a single zero eigenvalue.

	First we consider $\tau \Tr (H_n \varrho^\tau_{H_n})$. Since $x \exp(-x)$ is 
	a decreasing function for $x>1$ and since by assumption $\tau$ is constant 
	and $\lambda_{n-1}$ tends to infinity, we can bound
	\begin{equation}
	\begin{split}
	\tau \Tr (H_n \varrho^\tau_{H_n}) &\leq \frac{\tau (n-1)\lambda_{n-1} 
		\exp(-\tau 
		\lambda_{n-1})}{1+(n-1)\exp(-\tau \lambda_1)} \\
	&\leq  	\tau (n-1)\lambda_{n-1} \exp(-\tau \lambda_{n-1}).
	\end{split}
	\end{equation}
	Let $\lambda_{n-1}=\log(n)g(n)$, where $ g(n) \gg 1$.
	Then 
	\begin{equation}
	\tau (n-1) \lambda_{n-1} \exp(-\tau \lambda_{n-1}) 
	= \tau (n-1) \log (n) g(n) n^{-\tau g(n)}  
	\underset{n \rightarrow \infty}{\longrightarrow} 0.
	\end{equation}
	
	Now we bound
	\begin{equation}
	\begin{split}
	\log Z\leq  \sum_{i=1}^{n-1}\exp(-\tau\lambda_i) 
	\leq (n-1) \exp (-\tau \lambda_{n-1}).
	\end{split}
	\end{equation}
	If $\lambda_{n-1} \gg\log n$, then the formula above tends to 0. Since both 
	$\tau \Tr (H_n \varrho^\tau_{H_n})$ and $\log Z$ converge to zero we have the 
	result.
\end{proof}

\subsection{Proof of Theorem~10
}\label{app:th_7_proof}

\begin{proof}
	The entropy takes the form 
	\begin{equation}
	\begin{split}
	S(\varrho_{H_n}) 
	&= \tau \Tr (H_n \varrho^\tau_{H_n}) + \log Z \\
	&=\frac{\tau \sum_{i=1}^n\lambda_i \exp(-\tau \lambda_i)
	}{\sum_{i=1}^n\exp(-\tau \lambda_i)} +\log \left(\sum_{i=1}^n\exp(-\tau 
	\lambda_i)\right).
	\end{split}
	\end{equation}
	Since the matrix ${H_n}$
	is singular, we can extract single zero eigenvalue. 
	
	The $\log Z$ part can be bounded as 
	\begin{equation}
	\begin{split}
	\log Z  &\leq \log(1+(n-1)\exp(-\tau \lambda_{n-1})) \\
	&= \log(1-n^{-\tau a}+n^{1-\tau a}),
	\end{split}
	\end{equation}
	and
	\begin{equation}
	\begin{split}
	\log Z &\geq \log(1+(n-1)\exp(-\tau \lambda_1)) \\
	&= \log(1-n^{-\tau b}+n^{1-\tau b}).
	\end{split}
	\end{equation}
	Here behavior of $\log Z$ depends on $\tau$ parameter. If $\tau 
	<\frac{1}{b}$, then $\log Z \geq (1-\tau b)\log n +o(1)$ and 
	$\log Z \leq (1-\tau a)\log(n) +o(1)$. If $\tau >\frac{1}{a}$, 
	then $\log Z$ converges to 0.
	
	In the $\frac{1}{b}\leq \tau\leq\frac{1}{a}$ case we can provide partial 
	results only. 
	For $\tau=\frac{1}{b}$ we have 
	$\log Z \geq \log 2+o(1)$ and 
	$\log Z \leq(1-\frac{a}{b})\log n+o(1)$. 
	For $\tau =\frac{1}{a}$ we have
	$\log Z \leq \log 2+o(1)$. 
	For $\tau\in(\frac{1}{b},\frac{1}{a})$ we can only provide 
	$\log Z \leq (1-\tau a)\log n+o(1)$.
	
	Since $H_n$ is a nonnegative matrix, we have 
	$\tau \Tr (H_n \varrho^\tau_{H_n}) \geq0$. We 
	can again provide simple bounds
	\begin{equation}
	\begin{split}
	\tau \Tr (H_n \varrho^\tau_{H_n})&\leq \frac{\tau (n-1)\lambda_{n-1}\exp(-\tau 
		\lambda_{n-1})}{1+(n-1)\exp(-\tau \lambda_1)}\\
	& \leq \frac{\tau (n-1)a \log n\exp(-\tau a \log n)}{(n-1)\exp(-\tau b 
		\log n)} \\
	&  = \tau a n^{\tau(b-a)} \log n ,
	\end{split}
	\end{equation} 
	and similarly
	\begin{equation}
	\begin{split}
	\tau \Tr (H_n \varrho^\tau_{H_n}) &\geq \frac{\tau (n-1)\lambda_{1}\exp(-\tau 
		\lambda_{1})}{1+(n-1)\exp(-\tau \lambda_{n-1})}\\
	& \geq \frac{\tau (n-1)b \log n\exp(-\tau b \log n)}{n\exp(-\tau a \log 
		n)} \\
	& = \frac {n-1}{n} \tau b n^{\tau (a - b)}\log n.
	\end{split}
	\end{equation} 
	By combining the above inequalities we obtain the result.
\end{proof}

\subsection{Proof of Remark~14 
}\label{app:proof_rm_chung_lu_adjacency}
\begin{proof}
	Let $\lambda_n(-A) < 0$ be the single outlying eigenvalue of the matrix $-A$.
	By the use of Theorem 3 from \cite{chung2011spectra}  we have 
	the bound 
	\begin{equation}
	|\lambda_i(A)| \leq \sqrt{ 8 \omega_{\mathrm{max}} \log (\sqrt{2} n)}
	\end{equation}
	for $i = 1, \ldots , n-1$.
	From Lemma~1
	we note that 
	\begin{equation}
	S(\varrho_{A}) = S(\varrho_{-\lambda_n \1 +A})
	\end{equation}
	and therefore it suffices to consider the case of a shifted spectrum with 
	single zero eigenvalue and where for all the other eigenvalues we have 
	\begin{equation}
	\lambda_i(-\lambda_n \1 +A)  = \lambda_i(A) + \lambda_n(-A) 
	\geq \tilde{d} - 2\sqrt{ 8 \omega_{\mathrm{max}} \log (\sqrt{2} n)}.
	\end{equation}
	Using the assumption on $\tilde{d}$, asymptotically we obtain 
	$\lambda_i(-\lambda_n \1 +A) \gg\log n $ for 
	$i = 1, \ldots , n-1$.
	Then we use Theorem~7.%
\end{proof}

\section{Entropy of specific graph classes - 
	proofs}\label{app:spicific_graph_entropy_proofs}

The analytical spectra of all the graph classes discussed in this appendix are taken from~\cite{brouwer2011spectra}.

\subsection{Complete graph}
The Laplacian matrix of the complete graph has a single eigenvalue equal to zero
and $n-1$ eigenvalues equal to $n$. Therefore 
\begin{equation}
\begin{split}
S(\varrho_{L(K_n)}) 
&=\frac{\tau \sum_{i=1}^n\lambda_i \exp(-\tau \lambda_i)
}{\sum_{i=1}^n\exp(-\tau \lambda_i)} +\log \left(\sum_{i=1}^n\exp(-\tau 
\lambda_i)\right) \\
&= n  \tau \left( 1- \frac{1}{1+(n-1) \exp(-n \tau)} \right)
+ \log \left( 1+ (n-1) \exp(-n\tau)  \right)  \\
&= o(1).
\end{split}
\end{equation}
As the complete graph is a regular graph, then from
Proposition~3
we have $S(\varrho_{L(K_n)}) =
S(\varrho_{A(K_n)})$.
In the case of normalized Laplacian we use the fact that the complete graph is 
a $(n-1)$-regular graph. Therefore the spectrum of the normalized Laplacian 
consists of $n-1$ eigenvalues equal to $\frac{n}{n-1}$ and a single eigenvalue 
equal to $0$. Therefore we calculate
\begin{equation}
\begin{split}
S(\varrho_{\mathcal{L}(K_n)}) 
&= \tau \frac{n \exp\left( -\tau \frac{n}{n-1}\right)}{1+(n-1)\exp\left( -\tau 
	\frac{n}{n-1}\right)}
+ \log \left( 1+(n-1)\exp\left( -\tau \frac{n}{n-1}\right) \right) \\
&= \log n - o(1).
\end{split}
\end{equation}

\subsection{Complete bipartite graph}
Now we study entropy of the complete bipartite graph 
Let us set $|V|=n_1$ and $|W|=n_2$. The spectrum of the adjacency matrix of 
such a
complete bipartite graph $K_{n_1,n_2}$ consists of $n_1+n_2-2$ zero eigenvalues 
and 
$\pm\sqrt{n_1 n_2}$. Therefore we have
\begin{equation}
\begin{split}
S\left(\varrho_{A(K_{n_1,n_2})}\right) 
&= \tau \sqrt{n_1 n_2} \left( 1- \frac{2 \exp( \tau \sqrt{n_1n_2}) 
	+n_1+m_2-2}{\exp(- 
	\tau \sqrt{n_1n_2}) + \exp(\tau \sqrt{n_1n_2}) +n_1+n_2-2}   \right) \\
&+ \tau \sqrt{n_1n_2} + \log\left(1+ \exp(-2\tau\sqrt{n_1n_2}) 
+ \frac{n_1+n_2-2}{\exp(\tau \sqrt{n_1n_2})} \right) = o(1).
\end{split}
\end{equation}
The spectrum of Laplacian of the complete bipartite graph consists of a single 
$0$ eigenvalue, $n_1-1$ eigenvalues equal $n_2$, $n_2-1$ eigenvalues equal 
$n_1$ and a single $n_1+n_2$ eigenvalue. 
Now we assume $n_1=n_2$ and calculate
\begin{equation}
\begin{split}
S\left(\varrho_{L(K_{n_1,n_1})}\right) 
&=\tau n_1 \left(1- \frac{1-\exp(-2\tau n_1)}{1+2(n_1-1)\exp(-\tau n_1) + 
	\exp(-2 \tau n_1)} \right) \\
&+ \log\left( 1+2(n_1-1)\exp(-\tau n_1) + \exp(-2 \tau n_1) \right) =o(1).
\end{split}
\end{equation}
Assuming $n_2=1$ we obtain
\begin{equation}
\begin{split}
S\left(\varrho_{L(K_{n_1,1})}\right) 
&=\tau \left( 1-  \frac{1-n_1 \exp(-\tau(n_1+1))}{1+(n_1-1)\exp(-\tau) 
	+\exp(-\tau(n_1+1))}   \right) \\
&+\log \left(1+n_1\exp(-\tau) - \exp(-\tau)+\exp(-\tau(n_1+1)) \right)\\
&=\log (n_1+1) - o(1).
\end{split}
\end{equation}

Eigenvalues of a normalized Laplacian of a $K_{n_1,n_1}$ graph consist of 
single eigenvalues equal $0$ and $2$, and $2n_1-2$ eigenvalues equal $1$.
Therefore 
\begin{equation}
\begin{split}
S\left(\varrho_{\mathcal{L}(K_{n_1,n_1})}\right) 
&= \tau \left( 1-\frac{1-\exp(-2\tau)}{1+(2n_1-2)\exp(-\tau) + \exp(-2\tau)} 
\right) \\
&+\log \left( 1+ (2n_1-2)\exp(-\tau) + \exp(-2\tau)  \right) \\
&=\log(2n_1) - o(1).
\end{split}
\end{equation}

Eigenvalues of a normalized Laplacian of a star graph $K_{n_1,1}$ consist of a 
single $0$ eigenvalue, $n_1-1$ eigenvalues equal $1$ and a single eigenvalue 
equal $2$. 
Thus we have
\begin{equation}
\begin{split}
S\left(\varrho_{\mathcal{L}(K_{n_1,1})}\right) 
&= \tau \left( 1-\frac{1-\exp (-2 \tau)}{1+\exp(-2 \tau) + (n_1-1)\exp(-\tau)} 
\right) \\
&+ \log \left( 1+\exp(-2 \tau) + (n_1-1)\exp(-\tau) \right) \\
&= \log (n_1+1) - o(1).
\end{split}
\end{equation}

\subsection{Cycle graph}
Now we consider the cycle graph. We will prove Eq.~(12)
 from the main part of the article.
The eigenvalues of the adjacency matrix of the cycle $C_n$ take the form
$\lambda_j = 2\cos (\frac{2 \pi j}{n})$ for $j = 0, \ldots ,n-1$. 
Let $N_{\tau, n,j}:= \exp\left(- 2 \tau\cos \left( \frac{2 \pi j}{n} 
\right)\right)$. Then
\begin{equation}
\begin{split}
S(\varrho_{A(C_n)}) 
&=2\tau \frac{\sum_{j=0}^{n-1} \cos \left( \frac{2 \pi 
		j}{n} \right) N_{\tau, n,j}  }{\sum_{j=0}^{n-1}  
	N_{\tau, n,j}}
+ \log \left( \sum_{j=0}^{n-1}  N_{\tau, n,j} \right)  \\
&= 2\tau \frac{ \frac{1}{n}  \sum_{j=0}^{n-1} \cos \left( \frac{2 \pi 
		j}{n} \right) N_{\tau, n,j}  }{ \frac{1}{n} \sum_{j=0}^{n-1}  
	N_{\tau, n,j}}
+ \log \left( n \frac{1}{n} \sum_{j=0}^{n-1}  N_{\tau, n,j} \right) \\
&= 2\tau \frac{ \frac{1}{n}  \sum_{j=0}^{n-1} \cos \left( \frac{2 \pi 
		j}{n} \right) N_{\tau, n,j}  }{ \frac{1}{n} \sum_{j=0}^{n-1}  
	N_{\tau, n,j}}
+\log \left( \frac{1}{n} \sum_{j=0}^{n-1}  N_{\tau, n,j} \right) 
+ \log n.
\end{split}
\end{equation}
Now let us denote $x_j := \frac{j}{n}$. We calculate
\begin{equation}
\begin{split}
\frac{1}{n} \sum_{j=0}^{n-1}N_{\tau, n,j} 
&=\sum_{j=0}^{n-1} \frac{1}{n} \exp \left(-2 \tau \cos \left( 2 \pi x_j\right) 
\right)  \\
&\underset{n \to \infty}{\longrightarrow}
\int_0^1 \exp \left(-2 \tau \cos \left( 2 \pi x \right) \right) \dd x
= I_0 (2\tau),
\end{split}
\end{equation}
where $I_\alpha(x)$ is the modified Bessel function of the first kind.
Analogously we obtain
\begin{equation}
\begin{split}
\frac{1}{n} \sum_{j=0}^{n-1} \cos(2 \pi x_j) N_{\tau, n,j} 
&=\sum_{j=0}^{n-1}  \frac{1}{n} \cos(2 \pi x_j)  \exp \left(-2 \tau \cos \left( 
2 \pi x_j\right) \right) \\  
&\underset{n \to \infty}{\longrightarrow}
\int_0^1 \cos(2 \pi x)  \exp \left(-2 \tau \cos \left( 2 \pi x \right) \right) 
\dd x
= -I_1 (2\tau).
\end{split}
\end{equation}
Summing up, as
\begin{equation}
\begin{split}
&2\tau \frac{ \frac{1}{n}  \sum_{j=0}^{n-1} \cos \left( \frac{2 \pi 
		j}{n} \right) N_{\tau, n,j}  }{ \frac{1}{n} \sum_{j=0}^{n-1}  
	N_{\tau, n,j}}
+\log \left( \frac{1}{n} \sum_{j=0}^{n-1}  N_{\tau, n,j} \right) \\
&\underset{n \to \infty}{\longrightarrow}
2 \tau \frac{-I_1 (2\tau)}{I_0 (2 \tau)} +\log\left(I_0 (2\tau)\right),
\end{split}
\end{equation}
then for fixed $\tau$ we have
\begin{equation}
S(\varrho_{A(C_n)}) = \log n - 2 \tau \frac{I_1 (2\tau)}{I_0 (2 \tau)} 
+\log\left(I_0 (2\tau)\right) + o(1).
\end{equation}
As a cycle is a $2$-regular graph, then from Proposition~3 
we have that the same result will be obtained for the 
Laplacian matrix of a cycle.

To see why Eq. (13)
from the main part of the article holds we note that  
as a cycle is a $2$-regular graph, then  $\mathcal{L}(C_n) = \frac{1}{2}L(C_n)$.
Therefore it suffices to follow the proof of Eq. (12)
 from the main part of the article
knowing that the eigenvalues of the normalized Laplacian are 
$\lambda_j = 1-\cos  (\frac{2 \pi j}{n})$ for $j = 0, \ldots ,n-1$. 


\end{document}